\documentclass[12pt,reqno]{amsart}

\numberwithin{equation}{section}

\newtheorem{teo}{Theorem }[section]

\newtheorem{lem}[teo]{Lemma}

\newtheorem{defn}[teo]{Definition}

\begin{document}

%%%%%%%%%%%%%%%%%%%%%%%%%%%

\title[ Higher-Order Equations ]
      {Higher-order stationary dispersive equations on bounded intervals: a relation between the order of an equation and the growth of its convective term }
      
\author{N. A. Larkin$^\dag$\;	\&
J. Luchesi}

\address{Nikolai A. Larkin \newline
Departamento de Matem\'{a}tica, Universidade Estadual de Maring\'{a},
Av. Colombo 5790: Ag\^{e}ncia UEM, 87020-900, Maring\'{a}, PR, Brazil}
\email{nlarkine@uem.br}

\address{Jackson Luchesi \newline
Departamento de Matem\'{a}tica, Universidade Tecnol\'{o}gica Federal do Paran\'{a} - C\^{a}mpus Pato Branco, Via do Conhecimento Km 1, 85503-390, Pato Branco, PR, Brazil} \email{jacksonluchesi@utfpr.edu.br}

\thanks{}
\thanks{$^\dag$ N. A. Larkin was supported by Funda\c{c}\~ao Arauc\'{a}ria, Estado do Paran\'{a}, Brasil} \subjclass[]{}


\begin{abstract}
A boundary value problem for a stationary nonlinear dispersive equation of order $2l+1\;\; l\in \mathbb{N}$ with a convective term in the form $u^ku_x\;\; k\in \mathbb{N}$  was considered on an interval $(0,L)$. The existence, uniqueness and continuous dependence  of a regular solution as well as a relation between   $l$ and critical values of $k$  have been established.
	
\end{abstract}

\maketitle

\section{Introduction}\label{sec1}
This work concerns the existence, uniqueness and continuous dependence of regular
solutions to a boundary value problem for one class of nonlinear stationary dispersive  equations posed on bounded intervals
\begin{equation}\label{e1.1}
	au+\sum_{j=1}^l(-1)^{j+1}D^{2j+1}_x u +u^ku_x=f(x), \;\;l,k\in \mathbb{N},
\end{equation}
where $a$ is a positive constant. This class of stationary equations appears naturally while one wants to solve the corresponding evolution equation
\begin{equation}\label{e1.2}
	u_t+\sum_{j=1}^l(-1)^{j+1}D^{2j+1}_x u +u^ku_x=0, \;\;l,k\in\mathbb{N}
\end{equation}
making use of an implicit semi-discretization scheme:

\begin{equation}\label{e1.3}
	\frac{u^n-u^{n-1}}{h}+\sum_{j=1}^l(-1)^{j+1}D^{2j+1}_x u^n +(u^n)^ku^n_x=0, \;\;l\in\mathbb{N},
\end{equation}
where $h>0,$ \cite{Temam}. Comparing \eqref{e1.3}  with \eqref{e1.1}, it is clear that $a=\frac{1}{h}>0$ and $f(x)=\frac{ u^{n-1}}{h}.$ The case $k=1$ has been studied in \cite{larluch}. 
\par For $l=1$, we have the well-known generalized Korteweg-de Vries (KdV) equation which has been studied intensively for critical and supercritical values of $k$. In \cite{farah,lipaz,martel,merle} it was proved that a supercritical equation does not have global solutions and a critical one has a global solution for "small" initial data and the right-hand side. For $l=2, k=2$ the generalized  Kawahara equation has been studied in \cite{araruna}.
Initial value problems for
the Kawahara equation, $l=2$, which had been derived in \cite{kawa} as a perturbation of the KdV equation, have been considered in
\cite{biagioni,cui,farah,linponce,huo,kato,ponce2,ponce1,pilod,saut2} and attracted attention due to various applications of those
results in mechanics and physics such as dynamics of long
small-amplitude waves in various media \cite{hasim,jeffrey,kaku}.
On the other hand, last years appeared publications on solvability of
initial-boundary value problems for various dispersive equations (which
included the KdV and Kawahara equations) in bounded and unbounded domains
\cite{araruna,bona3,bubnov,colin,familark,khanal,kuvsh,larkin,larluch,marcio2}.
In spite of the fact that there is not some clear physical
interpretation for the problems on bounded intervals, their  study
is motivated by numerics \cite{chile}.
The KdV and Kawahara equations have been
developed for unbounded regions of wave propagations, however, if
one is interested in implementing numerical schemes to calculate
solutions in these regions, there arises the issue of cutting off a
spatial domain approximating unbounded domains by bounded ones. In
this case, some boundary conditions are needed to specify a
solution. Therefore, precise mathematical analysis of mixed problems
in bounded domains for dispersive equations is welcome and attracts
attention of specialists in this area
\cite{araruna,bona3,bubnov,colin,doronin2,
	familark,larkin}. \par As a
rule, simple boundary conditions at $x=0$ and $x=1$ such as
$u=u_x=0|_{x=0},\; u=u_x=u_{xx}=0|_{x=1}$ for the Kawahara equation were imposed.
Different kind
of boundary conditions was considered in \cite{ colin, lar2}.
Obviously,  boundary conditions for \eqref{e1.1} are the same as for \eqref{e1.2}. Because of that, study of boundary value problems for \eqref{e1.1} helps to understand solvability of initial- boundary value problems for \eqref{e1.2}.
\par Last years,  publications on dispersive equations of higher orders  appeared  \cite{familark,linponce,ponce2,ponce1,tao}. Here, we propose \eqref{e1.1} as a  stationary analog of \eqref{e1.2} because the last  equation includes classical models such as the generalized KdV and Kawahara equations.

The goal of our work is to formulate a correct boundary value problem for \eqref{e1.1} and to prove the existence, uniqueness and continuous dependence on perturbations of $f(x)$ for regular solutions as well as to study relation between the term $l$ of equation and the critical values of $k$. 

The paper has the following structure. Section 1 is Introduction.
Section 2 contains formulation of the problem and main results of the article. In Section 3 we give some useful facts. In Section 4 the existence of a regular solutions for the problem is proved. Here, a connection between the order of the equation and the growth of its convective term is established. Finally, in Section 5 uniqueness is proved provided certain restriction on $f$ as well as continuous dependence of solutions.

\section{Formulation of the Problem and Main Results}\label{sec2}
For real $a>0$, consider the following one-dimensional stationary higher order equation:
\begin{equation}\label{e2.1}
	au+\sum_{j=1}^{l}(-1)^{j+1}D^{2j+1}u+u^kDu=f\quad \text{in } (0,L)
\end{equation}
subject to boundary conditions:
\begin{equation} \label{e2.2}
	D^{i}u(0)=D^{i}u(L)=D^{l}u(L)=0,\,\, i=0,\ldots, l-1, \\
\end{equation}
where $0<L<\infty$, $l,k\in \mathbb{N}$ with $1\leq k \leq 4l$, $D^i={d^i}/{d x^i}$, $D^1\equiv D$  are the derivatives of order $i\in\mathbb{N}$, and $f\in L^2(0,L)$ is the given function.

Throughout this paper we adopt the usual notation $( \cdot , \cdot)$ for the inner product in $L^2(0,L)$ and $\| \cdot \|$, $\| \cdot \|_{\infty}$ and $\|\cdot \|_{H^i}$, $i\in\mathbb{N}$ for the norm in $L^2(0,L)$, $L^{\infty}(0,L)$ and $H^i(0,L)$, respectively \cite{adams}. Symbols $C_*$, $C_0$, $C_i$, $K_i$, $i\in \mathbb{N}$, mean positive constants appearing during the text.
\begin{defn} For a fixed $l\in \mathbb{N},$ equation \eqref{e2.1} is a regular one for $k<4l$ and is critical when $k=4l.$
\end{defn}
The main results of this article are the following theorems:

\begin{teo}\label{thm1} Let  $f\in L^2(0,L)$, then
	in the regular case, $1\leq k<4l$, problem \eqref{e2.1}-\eqref{e2.2} admits at least one regular solution $u\in H^{2l+1}(0,L)$ such that
	\begin{equation}\label{e2.3}
		\|u\|_{H^{2l+1}}\leq \mathcal{C} ((1+x),f^2)^{\frac{1}{2}}
	\end{equation}
	with the constant $\mathcal{C}$ depending only on $L$, $l$, $k$, $a$ and $((1+x),f^2)$.\\	
	In the critical case, $k=4l$, let $f$ be such that 
	\begin{equation}\label{e2.4}
		\|f\|<\frac{[(2l+1)(4l+2)]^{\frac{1}{4l}}a}{2^{\frac{1}{4l}}C_*}
	\end{equation}
	with $C_*$ an absolute constant. Then  problem \eqref{e2.1}-\eqref{e2.2} admits at least one regular solution $u\in H^{2l+1}(0,L)$ such that
	\begin{equation}\label{e2.5}
		\|u\|_{H^{2l+1}}\leq \mathcal{C}' ((1+x),f^2)^{\frac{1}{2}}
	\end{equation}
	with the constant $\mathcal{C}'$ depending only on $L$, $l$, $a$ and $((1+x),f^2)$.
\end{teo}
\begin{teo}\label{thm2} Let $l,\;k \in \mathbb{N}\;\;1 \leq k\leq 4l$ and let $((1+x),f^2)$ be sufficiently small. Then the solution from Theorem \ref{thm1} is unique and continuously depends on perturbations of $f$.
\end{teo}

\section{Preliminary Results}\label{sec3}
\begin{lem} \label{lem}
	For all $u\in H^1(0,L)$ such that $u(x_0)=0$ for some $x_0\in [0,L]$
	\begin{equation}\label{e3.1}
		\sup_{x\in(0,L)}|u(x)|\leq \sqrt{2}\|u\|^{\frac{1}{2}}\|Du\|^{\frac{1}{2}}.	
	\end{equation}
\end{lem}
\begin{proof}
	Let  $x_0\in [0,L]$ be such that $u(x_0)=0$. Then for any $x\in (0,L)$
	$$u^2(x)=\int_{x_0}^{x}D[u^2(\xi)]d\xi\leq 2\int_{x_0}^{x}|u(\xi)||D(\xi)|d\xi\leq 2\int_0^L |u(x)||Du(x)|dx$$$$\leq 2\|u\|\|Du\|.$$
	From this, \eqref{e3.1} follows immediately.
\end{proof}
We will use the following versions of the Gagliardo-Nirenberg's inequality, \cite{lady2, nirenberg, niren}.
\begin{teo}\label{thm3.2}
	Let $u$ belong to $H_0^l(0,L)$, then the following inequality holds:
	\begin{equation}\label{e3.2}
		\|u\|_{\infty}\leq C_*\|D^lu\|^{\frac{1}{2l}}\|u\|^{1-\frac{1}{2l}}
	\end{equation}
	with $C_*$ an absolute constant.
\end{teo}
\begin{teo}\label{thm3.3}
	Suppose $u$ and $D^{2l+1}u$ belong to $L^2(0,L)$. Then for the derivatives $D^iu$, $0\leq i<2l+1$ the following inequalities hold: 
	\begin{equation}\label{e3.3}
		\|D^iu\|_{L^p}\leq K_1\|D^{2l+1}u\|^{\theta}\|u\|^{1-\theta}+K_2\|u\|,
	\end{equation}
	where
	$$
	\frac{1}{p}=i-\theta(2l+1)+\frac{1}{2},
	$$
	for all $\theta \in [\frac{i}{2l+1}, 1] $. (The constants $K_1$, $K_2$ depend only on $L$, $l$, $i$).
\end{teo}
We will use  the following  fixed point theorem, \cite{evans}.
\begin{teo}\label{thm3.4}(Schaefer's Fixed Point Theorem)
	Let $X$ a real Banach Space. Suppose $B:X\rightarrow X$ is a compact and continuous  mapping. Assume further that the set 
	$$\{u\in X\,\,|\,\,u=\lambda Bu \,\,\mbox{for some}\,\, 0\leq \lambda\leq 1\}$$
	is bounded. Then $B$ has a fixed point.
\end{teo}

\section{Existence}\label{sec4}
\begin{proof}(of Theorem \ref{thm1}).
	
	We start with the linearized version of \eqref{e2.1}
	\begin{equation}\label{e4.1}
		Au\equiv au+\sum_{j=1}^{l}(-1)^{j+1}D^{2j+1}u=f\quad \text{in } (0,L)
	\end{equation}
	subject to boundary conditions \eqref{e2.2}.
	\begin{teo}\label{thm4.1}(See \cite{larluch}, Theorem 5). Let $f\in L^2(0,L)$. Then the problem \eqref{e4.1},\eqref{e2.2} admits a unique regular solution $u\in H^{2l+1}(0,L)$ such that
		\begin{equation}\label{e4.2}
			\|u\|_{H^{2l+1}}\leq C_0 \|f\|
		\end{equation}
		with the constant $C_0$ depending only on $L$ and $a$.
	\end{teo}
	Given $u\in H_0^l(0,L)$, set $F:=f-u^kDu$. By \eqref{e3.2}, we get
	\begin{eqnarray}\label{e5.1}
		\|F\|& \leq & \|f\|+\|u^kDu\|\leq\|f\|+\|u\|_{\infty}^k\|Du\|\notag\\ & \leq & \|f\|+C_*^k\|u\|^{\left(1-\frac{1}{2l}\right)k}\|D^lu\|^{\frac{k}{2l}}\|Du\|\notag\\ & \leq & \|f\|+C_*^k\|u\|_{H_0^l}^{\left(1-\frac{1}{2l}\right)k}\|u\|_{H_0^l}^{\frac{k}{2l}}\|u\|_{H_0^l}\notag\\
		& \leq & \|f\|+C_*^k\|u\|_{H_0^l}^{k+1}.
	\end{eqnarray}
	By Theorem \ref{thm4.1}, let $w\in H^{2l+1}(0,L)$ be a unique solution of the linear equation 
	\begin{equation}\label{e5.2}
		aw+\sum_{j=1}^{l}(-1)^{j+1}D^{2j+1}w=F\quad \text{in } (0,L)
	\end{equation}
	subject to boundary conditions \eqref{e2.2}.
	By \eqref{e4.2}-\eqref{e5.1}, 
	\begin{equation}\label{e5.3}
		\|w\|_{H^{2l+1}}\leq C_0\|F\|\leq C_0(\|f\|+C_*^k\|u\|_{H_0^l}^{k+1}).
	\end{equation}
	We will write henceforth $Bu=w$ whenever $w$ is derived from $u$ via \eqref{e5.2},\eqref{e2.2}, that is, $Bu\equiv A^{-1}(F(u))$, where $A$ is defined by \eqref{e4.1}.
	\begin{lem}\label{lem1}
		The mapping $B:H_0^l(0,L)\rightarrow H_0^l(0,L)$ is compact and continuous.
	\end{lem}
	\begin{proof}
		Indeed, if $\{u_n\}$ is a bounded sequence in $H_0^l(0,L)$, then in view of estimate \eqref{e5.3}, the sequence $\{w_n\}$, where $w_n=Bu_n$, $n\in\mathbb{N}$ is bounded in $H^{2l+1}(0,L)$. Since $H^{2l+1}(0,L)$ is compactly embedded in $H_0^l(0,L)$, there exists a convergent in $H_0^l(0,L)$ subsequence $\{Bu_{n_m}\}_{m=1}^{\infty}$, therefore $B$ is compact.

		To prove continuity of the mapping $B$, let $\{u_n\}$ be a sequence such that $u_n \rightarrow u$ in $H_0^l(0,L)$. Then the difference $v_n=w_n-w$, where $w_n=Bu_n$, $n\in\mathbb{N}$ and $w=Bu$ satisfies
		\begin{equation}\label{e5.4}
			av_n+\sum_{j=1}^{l}(-1)^{j+1}D^{2j+1}v_n=u^kD(u-u_n)+(u^k-u_n^k)Du_n
		\end{equation}
		and the boundary conditions \eqref{e2.2}.
		
		Multiplying \eqref{e5.4} by $v_n$ and integrating by parts over $(0,L)$, we obtain
		$$a\|v_n\|^2+\frac{1}{2}(D^lv_n(0))^2=(u^kD(u-u_n)+(u^k-u_n^k)Du_n,v_n),$$
		whence
		\begin{equation}\label{e5.5}
			a\|v_n\|\leq \|u^kD(u-u_n)\|+\|(u^k-u_n^k)Du_n\|.
		\end{equation}
		According to \eqref{e3.1},
		\begin{eqnarray*}\|u^kD(u-u_n)\| & \leq & \left(\sup_{x\in (0,L)}|u(x)|^{2k}\right)^{\frac{1}{2}}\|D(u_n-u)\| \\ & \leq & 2^{\frac{k}{2}}\|u\|^{\frac{k}{2}}\|Du\|^{\frac{k}{2}}\|u_n-u\|_{H_0^l}\\ & \leq & 2^{\frac{k}{2}}\|u\|_{H_0^l}^k \|u_n-u\|_{H_0^l} \rightarrow 0
		\end{eqnarray*}
		because $u_n \rightarrow u$ in $H_0^l(0,L)$.
		On the other hand, let $g\in C^1(\mathbb{R})$ be such that $g(y)=y^k$. By the Mean Value Theorem, for arbitrary $y,z\in\mathbb{R}$ there is $\xi\in (y,z)$ such that
		$$|y^k-z^k|=k\xi^{k-1}|y-z|.$$
		Since $\xi\in (y,z)$ we can write $\xi=(1-\tau)y+\tau z$, with $\tau\in(0,1)$. Taking $y=u_n(x)$ and $z=u(x)$ for each $x\in(0,L)$, we obtain
		\begin{eqnarray}\label{e5.6}
			|u_n^k(x)-u^k(x)|^2 & = & k^2|(1-\tau)u_n(x)+\tau u(x)|^{2(k-1)}|u_n(x)-u(x)|^2 \notag\\
			& \leq & k^2[|1-\tau||u_n(x)|+|\tau| |u(x)|]^{2(k-1)}|u_n(x)-u(x)|^2 \notag\\
			& \leq & k^2[|u_n(x)|+|u(x)|]^{2(k-1)}|u_n(x)-u(x)|^2 \notag\\
			& \leq & k^22^{2(k-1)}|u_n(x)|^{2(k-1)}|u_n(x)-u(x)|^2\notag\\
			& + & 	k^22^{2(k-1)}|u(x)|^{2(k-1)}|u_n(x)-u(x)|^2.
		\end{eqnarray}
		By \eqref{e3.1},
		$$\sup_{x\in (0,L)}|u_n(x)|^{2(k-1)}\leq 2^{k-1}\|u_n\|^{k-1}\|Du_n\|^{k-1}\leq 2^{k-1}\|u_n\|_{H_0^l}^{2(k-1)},$$
		$$\sup_{x\in (0,L)}|u(x)|^{2(k-1)}\leq 2^{k-1}\|u\|^{k-1}\|Du\|^{k-1}\leq 2^{k-1}\|u\|_{H_0^l}^{2(k-1)}$$
		and
		$$\sup_{x\in (0,L)}|u_n(x)-u(x)|^2\leq 2\|u_n-u\|\|D(u_n-u)\|\leq 2\|u_n-u\|_{H_0^l}^2.$$
		Thus
		\begin{eqnarray*}
			\|(u^k-u_n^k)Du_n\| & \leq & \left(\sup_{x\in (0,L)}|u_n^k(x)-u^k(x)|^2\right)^{\frac{1}{2}}\|Du_n\| \\
			& \leq & k2^{\frac{3k-2}{2}}(\|u_n\|_{H_0^l}^{k-1}+\|u\|_{H_0^l}^{k-1})^{\frac{1}{2}}\|u_n-u\|_{H_0^l} \rightarrow 0	
		\end{eqnarray*}
		because the sequence $\{u_n\}$ is bounded in $H_0^l(0,L)$ and  $u_n \rightarrow u$ in $H_0^l(0,L)$. From \eqref{e5.5}, we conclude that $\|v_n\|\rightarrow 0$.

		Multiplying \eqref{e5.4} by $(1+x)v_n$ and integrating over $(0,L)$, we obtain
		\begin{eqnarray*}
			a(v_n,(1+x)v_n)+\sum_{j=1}^{l}(-1)^{j+1}(D^{2j+1}v_n,(1+x)v_n) &\\=(u^kD(u-u_n)+(u^k-u_n^k)Du_n,(1+x)v_n).
		\end{eqnarray*}
		Integrating by parts and using \eqref{e2.2} it follow that
		\begin{eqnarray*}
			a\|v_n\|^2+\sum_{j=1}^{l}\left(\frac{2j+1}{2}\right)\|D^jv_n\|^2+\frac{1}{2}(D^lv_n(0))^2 &\\
			\leq (\|u^kD(u-u_n)\|+\|(u^k-u_n^k)Du_n\|)\|(1+x)v_n\|.
		\end{eqnarray*}
		Since $\|u^kD(u-u_n)\|, \|(u^k-u_n^k)Du_n\|, \|v_n\|\rightarrow 0$, we get $\|v_n\|_{H_0^l}\rightarrow 0$, that is, $w_n \rightarrow w$ in $H_0^l(0,L)$. Hence, $u_n \rightarrow u$ in $H_0^l(0,L)$ implies $Bu_n\rightarrow Bu$ in $H_0^l(0,L)$. This proves that $B$ is continuous. 
	\end{proof}
	\begin{lem}\label{lem2}
		The set 
		$$\{u\in H_0^l(0,L)\,\,|\,\,u=\lambda Bu \,\,\mbox{for some}\,\, 0\leq \lambda\leq 1\}$$
		is bounded in $H_0^l(0,L)\cap H^{2l+1}(0,L)$.
	\end{lem}
	\begin{proof} Assume $u\in H_0^l(0,L)$ such that 
		$$u=\lambda Bu \,\,\mbox{for some}\,\, 0< \lambda\leq 1,$$
		then 
		$$
		a\left(\frac{u}{\lambda}\right)+\sum_{j=1}^{l}(-1)^{j+1}D^{2j+1}\left(\frac{u}{\lambda}\right)=f-u^kDu \quad \text{in } (0,L)
		$$
		and
		$$D^{i}\left(\frac{u}{\lambda}\right)(0)=D^{i}\left(\frac{u}{\lambda}\right)(L)=D^{l}\left(\frac{u}{\lambda}\right)(L)=0,\,\, i=0,\ldots, l-1, 
		$$
		that is
		\begin{equation}\label{e5.7}
			au+\sum_{j=1}^{l}(-1)^{j+1}D^{2j+1}u+\lambda u^kDu=\lambda f \quad \text{in } (0,L)
		\end{equation}
		and $u$ satisfies the boundary conditions \eqref{e2.2}. 
		
		To prove this Lemma, we need some a priori estimates:
		
		\subsection*{Estimate I:}
		Multiplying \eqref{e5.7} by $u$ and integrating over $(0,L)$, we
		obtain
		\begin{equation}\label{e5.8}
			a\|u\|^2+\sum_{j=1}^{l}(-1)^{j+1}(D^{2j+1}u,u)+\lambda(u^kDu,u)=(\lambda f,u).
		\end{equation}
		Integrating by parts and using \eqref{e2.2}, we get
		$$\lambda (u^kDu,u)=0$$
		and
		$$\sum_{j=1}^{l}(-1)^{j+1}(D^{2j+1}u,u)=\frac{1}{2}(D^lu(0))^2.$$
		Thus \eqref{e5.8} becomes
		$$a\|u\|^2+\frac{1}{2}(D^lu(0))^2=(\lambda f,u)$$
		and
		\begin{equation}\label{e5.9}
			\|u\|\leq \frac{1}{a}\|f\|.
		\end{equation}
		
		\subsection*{Estimate II:}
		Multiplying \eqref{e5.7} by $(1+x)u$ and
		integrating over $(0,L)$, we obtain
		\begin{align}\label{e5.10}
			a(u,(1+x)u)+\sum_{j=1}^{l}(-1)^{j+1}(D^{2j+1}u,(1+x)u)\notag &\\+\lambda(u^kDu,(1+x)u)=(\lambda f,(1+x)u).
		\end{align}
		Since
		$$
		\sum_{j=1}^{l}(-1)^{j+1}(D^{2j+1}u,(1+x)u)=\sum_{j=1}^{l}\left(\frac{2j+1}{2}\right)\|D^ju\|^2+\frac{1}{2}(D^lu(0))^2,
		$$
		integrating by parts and using \eqref{e2.2},\eqref{e3.2}, we get
		\begin{eqnarray}\label{e5.11}
			\lambda(u^kDu,(1+x)u) & = & \lambda(u^kDu,xu)=\frac{\lambda}{k+2}\int_{0}^{L}xD[u^{k+2}]dx \notag \\
			& = & -\frac{\lambda}{k+2}\int_{0}^{L}u^{k+2}dx\leq \frac{1}{k+2}\|u\|_{\infty}^k\|u\|^2 \notag\\
			& \leq & \underbrace{\frac{C_*^k}{k+2}\|u\|^{2+\left(\frac{2l-1}{2l}\right)k}\|D^lu\|^{\frac{k}{2l}}}_I.
		\end{eqnarray}
		\begin{center}
			\bf{Regular case $\mathbf{1\leq k<4l}$.}
		\end{center}	
		By the Young inequality, with $p=\frac{4l}{k}$, $q=\frac{4l}{4l-k}$ and arbitrary $\epsilon_1>0$,
		$$
		I\leq\epsilon_1\frac{k}{4l}\|D^lu\|^2+\frac{1}{\epsilon_1^{\frac{k}{4l-k}}}\left(\frac{4l-k}{4l}\right)\left(\frac{C_*^k}{k+2}\right)^{\frac{4l}{4l-k}}\|u\|^{\frac{8l+(4l-2)k}{4l-k}}.
		$$
		Again by the Young inequality with arbitrary $\epsilon_2>0$,
		$$
		(f,(1+x)u)\leq \frac{\epsilon_2}{2}((1+x),u^2)+ \frac{1}{2\epsilon_2}((1+x),f^2).
		$$
		Therefore, \eqref{e5.10} reduces to the inequality
		$$
		\begin{array}{l}
		\left(a-\frac{\epsilon_2}{2}\right)((1+x),u^2)+\sum_{j=1}^{l-1}\left(\frac{2j+1}{2}\right)\|D^ju\|^2+\left(\frac{2l+1}{2}-\epsilon_1\frac{k}{4l}\right)\|D^lu\|^2 
		\\ \leq \frac{1}{\epsilon_1^{\frac{k}{4l-k}}}\left(\frac{4l-k}{4l}\right)\left(\frac{C_*^k}{k+2}\right)^{\frac{4l}{4l-k}}\|u\|^{\frac{8l+(4l-2)k}{4l-k}}+ \frac{1}{2\epsilon_2}((1+x),f^2).
		\end{array}
		$$
		Taking $\epsilon_1=\frac{4l(2l-1)}{2k}>0$ and $\epsilon_2=a>0$, we get
		
		\begin{align}\label{e5.12}
			\frac{a}{2}((1+x),u^2)+\sum_{j=1}^{l-1}\left(\frac{2j+1}{2}\right)\|D^ju\|^2+\|D^lu\|^2\notag &\\ 
			\leq C_1 \|u\|^{\frac{8l+(4l-2)k}{4l-k}}+ \frac{1}{2a}((1+x),f^2),
		\end{align}
		where $$C_1=\left(\frac{2k}{4l(2l-1)}\right)^{\frac{k}{4l-k}}\left(\frac{4l-k}{4l}\right)\left(\frac{C_*^k}{k+2}\right)^{\frac{4l}{4l-k}}.$$
		Since
		$$((1+x),f^2)=\|f\|^2+(x,f^2)\geq\|f\|^2,$$
		it follows from \eqref{e5.9} that
		$$
		\|u\|^{\frac{8l+(4l-2)k}{4l-k}}\leq \left(\frac{1}{a}\right)^{\frac{8l+(4l-2)k}{4l-k}} ((1+x),f^2)^{\frac{4l+(2l-1)k}{4l-k}}
		$$
		and \eqref{e5.12} implies
		\begin{equation}\label{e5.13}
			\|u\|_{H_0^l}\leq C_2((1+x),f^2)^{\frac{1}{2}},
		\end{equation}
		where 
		$$
		C_2=\frac{1}{\sqrt{\beta}}\left[C_3((1+x),f^2)^{\frac{2lk}{4l-k}}+\frac{1}{2a}\right]^{\frac{1}{2}}
		$$
		with $\beta=\min\{\frac{a}{2},1\}$ and $C_3=C_1a^{-\frac{8l+(4l-2)k}{4l-k}}$.
		
		Rewriting  \eqref{e5.7} in the form
		$$(-1)^{l+1}D^{2l+1}u=\lambda f-au-\sum_{j=1}^{l-1}(-1)^{j+1}D^{2j+1}u-\lambda u^kDu,$$
		we estimate
		\begin{equation}\label{e1}
			\|D^{2l+1}u\|\leq\|f\|+a\|u\|+\sum_{j=1}^{l-1}\|D^{2j+1}u\|+\|u^kDu\|.
		\end{equation}
		For $l=1$ we have $\sum_{j=1}^{l-1}(-1)^{j+1}D^{2j+1}u=0$ and for $l\geq 2$ denote $J=\{1,\ldots, l-1\}$ and 
		$$
		\begin{array}{cc}
		I_1=\{j\in J|\,\, 2j+1\leq l\}, &  I_2=\{j\in J|\,\,l< 2j+1< 2l+1\}.\\
		\end{array}
		$$
		Hence we can write
		\begin{eqnarray}\label{e5.14}
			\|D^{2l+1}u\| & \leq & \|f\|+a\|u\|+\sum_{j\in I_1}\|D^{2j+1}u\|\notag\\
			& + & \sum_{j\in I_2}\|D^{2j+1}u\|+\|u^kDu\|.
		\end{eqnarray}
		By \eqref{e5.13},
		\begin{equation}\label{e5.16}
			a\|u\|+\sum_{j\in I_1}\|D^{2j+1}u\|\leq (a+l)C_2((1+x),f^2)^{\frac{1}{2}}
		\end{equation}
		and by \eqref{e3.2},\eqref{e5.13},
		\begin{equation}\label{e5.17}
			\|u^kDu\| \leq \|u\|_{\infty}^k\|Du\| \leq C_*^k\|u\|_{H_0^1}^{k+1}\leq C_*^k C_2^{k+1}((1+x),f^2)^{\frac{k+1}{2}}.
		\end{equation}
		On the other hand, $l<2j+1< 2l+1$ for all $j\in I_2$. Hence, by \eqref{e3.3}, there are $K_1^j$, $K_2^j$, depending only on $L$ and $l$, such that
		$$
		\|D^{2j+1}u\|\leq K_1^j\|D^{2l+1}u\|^{\theta^j}\|u\|^{1-\theta^j}+K_2^j\|u\| \quad\mbox{with}\quad \theta^j=\frac{2j+1}{2l+1}.
		$$
		Making use of Young's inequality with $p^j=\frac{1}{\theta^j}$, $q^j=\frac{1}{1-\theta^j}$ and arbitrary $\epsilon>0$, we get
		$$
		\|D^{2j+1}u\|\leq \epsilon\|D^{2l+1}u\|+C_4^j(\epsilon)\|u\|+K_2^j\|u\|, 
		$$
		where $C_4^j(\epsilon)=\left[q^j\left(\frac{p^j\epsilon}{({K_1^j})^{p^j}}\right)^{\frac{q^j}{p^j}}
		\right]^{-1}$. Summing over $j\in I_2$ and making use of \eqref{e5.9}, we find
		\begin{equation}\label{e5.15}
			\sum_{j\in I_2}\|D^{2j+1}u\|\leq l\epsilon\|D^{2l+1}u\|+\left(\frac{1}{a}\sum_{j\in I_2}(C_4^j(\epsilon)+K_2^j)\right)\|f\|.
		\end{equation}
		Substituing \eqref{e5.16},\eqref{e5.17} and \eqref{e5.15} into \eqref{e5.14}, we obtain 
		\begin{eqnarray*}
			\|D^{2l+1}u\| & \leq & l\epsilon \|D^{2l+1}u\|+\left(\frac{1}{a}\sum_{j\in I_2}(C_4^j(\epsilon)+K_2^j)\right)((1+x),f^2)^{\frac{1}{2}} \\
			& + & 
			\left(1+(a+l)C_2+C_*^kC_2^{k+1}((1+x),f^2)^{\frac{k}{2}}\right)((1+x),f^2)^{\frac{1}{2}}.
		\end{eqnarray*}
		Taking $\epsilon=\frac{1}{2l}$, we conclude
		\begin{equation}\label{e5.18}
			\|D^{2l+1}u\|\leq   C_5((1+x),f^2)^{\frac{1}{2}},
		\end{equation}
		where $C_5$ depends only on $L$, $l$, $k$, $a$ and $((1+x),f^2)$.
		
		Again by \eqref{e3.3}, for all $i=l+1, \ldots, 2l$, there are $K_1^i$, $K_2^i$ depending only on $L$ and $l$ such that
		$$
		\|D^iu\|\leq K_1^i\|D^{2l+1}u\|^{\theta^i}\|u\|^{1-\theta^i}+K_2^i\|u\| \quad\mbox{with}\quad \theta^i=\frac{i}{2l+1}.
		$$
		Making use of \eqref{e5.9} and \eqref{e5.18}, we get
		\begin{equation}\label{e5.19}
			\|D^iu\|\leq \left(\frac{K_1^i C_5^{\theta^i}}{a^{1-\theta^i}}+\frac{K_2^i}{a}\right)((1+x),f^2)^{\frac{1}{2}}, \quad i=l+1,\ldots, 2l.
		\end{equation}
		Taking into account \eqref{e5.13}, \eqref{e5.18} and \eqref{e5.19}, we obtain \eqref{e2.3}, that is
		\begin{equation*}
			\|u\|_{H^{2l+1}}\leq \mathcal{C}((1+x),f^2)^{\frac{1}{2}}
		\end{equation*}
		with $\mathcal{C}$ depending only on $L$, $l$, $k$, $a$ and $((1+x),f^2)$.
		\begin{center}
			{\bf Critical case $\mathbf{k=4l}$.}
		\end{center} 
		Returning to \eqref{e5.11}, we find
		$$I=\frac{C_*^{4l}}{4l+2}\|u\|^{4l}\|D^lu\|^2\leq \frac{C_*^{4l}}{(4l+2)a^{4l}}\|f\|^{4l}\|D^lu\|^2.
		$$
		Since
		$$
		(f,(1+x)u)\leq \frac{a}{2}((1+x),u^2)+ \frac{1}{2a}((1+x),f^2),
		$$
		we transform \eqref{e5.10} as follows
		\begin{eqnarray*}
			\frac{a}{2}\|u\|^2+\sum_{j=1}^{l-1}\left(\frac{2j+1}{2}\right)\|D^ju\|^2+\left(\frac{2l+1}{2}-\frac{C_*^{4l}}{(4l+2)a^{4l}}\|f\|^{4l}\right)\|D^lu\|^2 \\ 
			+\frac{1}{2}(D^lu(0))^2 \leq \frac{1}{2a}((1+x),f^2).
		\end{eqnarray*}
		For fixed $l$, $a$ and $f\in L^2(0,L)$ such that 
		$$\|f\|<\frac{[(2l+1)(4l+2)]^{\frac{1}{4l}}a}{2^{\frac{1}{4l}}C_*},$$
		we obtain
		$$\frac{2l+1}{2}-\frac{C_*^{4l}}{(4l+2)a^{4l}}\|f\|^{4l}>0.$$
		Therefore
		\begin{equation}\label{e5.20}
			\|u\|_{H_0^l}\leq \frac{1}{\sqrt{2a\gamma_l}}((1+x),f^2)^{\frac{1}{2}}
		\end{equation}
		with $\gamma_l=\min\{\frac{a}{2},\frac{3}{2},\frac{2l+1}{2}-\frac{C_*^{4l}}{(4l+2)a^{4l}}\|f\|^{4l}\}$.
		Retunrning to \eqref{e5.7} and acting as in the regular case with \eqref{e5.20}, we conclude \eqref{e2.5}, that is
		\begin{equation*}
			\|u\|_{H^{2l+1}}\leq \mathcal{C}'((1+x),f^2)^{\frac{1}{2}}
		\end{equation*}
		with $\mathcal{C}'$ depending only on $L$, $l$, $a$ and $((1+x),f^2)$.
	\end{proof}
	Applying Theorem \ref{thm3.4}, we complete the proof of the Theorem \ref{thm1}.
\end{proof}

\section{Uniqueness and Continuous Dependence}
\begin{proof}(of Theorem \ref{thm2}).
	
	We separated  two cases: $l\geq 2$ and $l=1.$\\
	\textbf{For $\mathbf{l\geq 2}$},
	let $u_1$ and $u_2$ be two distinct solutions of \eqref{e2.1}-\eqref{e2.2}. Then the difference $w=u_1-u_2$ satisfies the equation
	\begin{equation}\label{e6.1}
		aw+\sum_{j=1}^{l}(-1)^{j+1}D^{2j+1}w+u_1^kDw+(u_1^k-u_2^k)Du_2=0
	\end{equation}
	and the boundary conditions \eqref{e2.2}.
	
	Multiplying \eqref{e6.1} by $w$ and integrating over $(0,L)$, we obtain
	\begin{equation}\label{e6.2}
		a\|w\|^2+\frac{1}{2}(D^lw(0))^2+\underbrace{(u_1^kDw,w)}_{I_1}+\underbrace{((u_1^k-u_2^k)Du_2,w)}_{I_2}=0.
	\end{equation}
	Integrating by parts and using \eqref{e2.2},\eqref{e3.1},  we get
	\begin{eqnarray*}
		I_1 & = & -\frac{1}{2}\int_{0}^{L}w^2(x)Du_1^k(x)dx\leq \frac{k}{2}\int_{0}^{L}|u_1(x)|^{k-1}|Du_1(x)||w(x)|^2dx\\
		& \leq & \frac{k}{2}\sup_{x\in(0,L)}|u_1(x)|^{k-1}\sup_{x\in(0,L)}|Du_1(x)|\|w\|^2\\ & \leq &  k2^{\frac{k-2}{2}}\|u_1\|_{H_0^l}^k\|w\|^2.
	\end{eqnarray*}
	By \eqref{e3.1},\eqref{e5.6}, we have 
	\begin{eqnarray*}
		|I_2| & \leq & 	\int_{0}^{L}|u_1^k(x)-u_2^k(x)||Du_2(x)||w(x)|dx \\
		& \leq & k2^{k-1}\sup_{x\in(0,L)}|Du_2(x)|\int_{0}^{L}(|u_1(x)|^{k-1}+|u_2(x)|^{k-1})|w(x)|^2dx\\ 
		& \leq & k2^{\frac{2k-1}{2}}\|u_2\|_{H_0^l}\sup_{x\in(0,L)}\{|u_1(x)|^{k-1}+|u_2(x)|^{k-1}\}\|w\|^2\\
		& \leq & k2^{\frac{3k-2}{2}}\|u_2\|_{H_0^l}(\|u_1\|_{H_0^l}^{k-1}+\|u_2\|_{H_0^l}^{k-1})\|w\|^2.
	\end{eqnarray*}	
	Substituting $I_1,I_2$ into \eqref{e6.2}, we reduce it to the inequality
	\begin{equation}\label{e6.3}
		\left(a-k2^{\frac{k-2}{2}}\|u_1\|_{H_0^l}^k-k2^{\frac{3k-2}{2}}\|u_2\|_{H_0^l}(\|u_1\|_{H_0^l}^{k-1}+\|u_2\|_{H_0^l}^{k-1})\right)\|w\|^2\leq 0.
	\end{equation}
	\begin{center}
		\bf{Regular case $\mathbf{1\leq k<4l}$.}
	\end{center}
	Making use of \eqref{e5.13}, we can estimate \eqref{e6.3} as
	\begin{equation}\label{e6.4}
		\left(a-(2{^\frac{k-2}{2}}+2^{\frac{3k}{2}})kC_2^k((1+x),f^2)^{\frac{k}{2}}\right)\|w\|^2\leq 0,
	\end{equation}
	where
	$$C_2=\frac{1}{\sqrt{\beta}}\left[C_3((1+x),f^2)^{\frac{2lk}{4l-k}}+\frac{1}{2a}\right]^{\frac{1}{2}}
	$$
	with $\beta=\min\{\frac{a}{2},1\}$ and $C_3$ depending only on $l$, $k$ and $a$.
	For fixed $l$, $k$ and $a$, assume that 
	\begin{equation}\label{e6.5}
		((1+x),f^2)^{\frac{1}{2}}<\min\left\{\left(\frac{1}{2aC_3}\right)^{\frac{4l-k}{4lk}}, \frac{a^{\frac{1}{k}}}{[(2^{\frac{k-2}{2}}+2^{\frac{3k}{2}})k]^{\frac{1}{k}}(a\beta)^{-\frac{1}{2}}}\right\}.
	\end{equation}
	Then $C_2<\left(\frac{1}{a\beta}\right)^{\frac{1}{2}}$ and consequently
	$$
	\left(a-(2{^\frac{k-2}{2}}+2^{\frac{3k}{2}})kC_2^k((1+x),f^2)^{\frac{k}{2}}\right)> 0.
	$$
	Hence \eqref{e6.4} implies $\|w\|= 0$ and uniqueness is proved for $l\geq 2$ and $1\leq k<4l$.
	\begin{center}
		\bf{Critical case $\mathbf{k=4l}$.}
	\end{center}
	Rewrite \eqref{e6.3} in the form:
	$$
	\left(a-l2^{2l+1}\|u_1\|_{H_0^l}^{4l}-l2^{6l+1}\|u_2\|_{H_0^l}(\|u_1\|_{H_0^l}^{4l-1}+\|u_2\|_{H_0^l}^{4l-1})\right)\|w\|^2\leq 0.
	$$
	Making use of \eqref{e5.20}, we obtain
	$$
	\left(a-l(2^{2l+1}+2^{6l+2})\left(\frac{1}{2a\gamma_l}\right)^{2l}((1+x),f^2)^{2l}\right)\|w\|^2\leq 0,
	$$
	where
	$$\gamma_l=\min\left\{\frac{a}{2},\frac{3}{2},\frac{2l+1}{2}-\frac{C_*^{4l}}{(4l+2)a^{4l}}\|f\|^{4l}\right\}.$$
	For fixed $l$ and $a,$ suppose that
	\begin{equation}\label{e6.6}
		((1+x),f^2)^{\frac{1}{2}}<\min\left\{\frac{[(2l+1)(4l+2)]^{\frac{1}{4l}}a}{2^{\frac{1}{4l}}C_*}, \left(\frac{a}{\eta}\right)^{\frac{1}{4l}}\right\},
	\end{equation}
	where $\eta=l(2^{2l+1}+2^{6l+2})(2a\gamma_l)^{-2l}$. Since $\|f\|\leq ((1+x),f^2)^{\frac{1}{2}},$ it follows that \eqref{e2.4} is satisfied
	and
	$$\left(a-l(2^{2l+1}+2^{6l+2})\left(\frac{1}{2a\gamma_l}\right)^{2l}((1+x),f^2)^{2l}\right)>0.$$
	
	Thus $\|w\|=0$ and uniqueness is proved for $l\geq 2$ and $k=4l$. \\
	{\bf The case $\mathbf{l=1}$.}
	
	The problem \eqref{e2.1}-\eqref{e2.2} becomes:
	\begin{equation}\label{e6.7}
		au+D^3u+u^kDu=f\quad \text{in } (0,L),
	\end{equation}
	\begin{equation}\label{e6.8}
		u(0)=u(L)=Du(L)=0.
	\end{equation}
	Let $u_1$ and $u_2$ be two distinct solutions of \eqref{e6.7}-\eqref{e6.8}. Then the difference $w=u_1-u_2$ satisfies the equation
	\begin{equation}\label{e6.9}
		aw+D^3w+u_1^kDw+(u_1^k-u_2^k)Du_2=0
	\end{equation}
	and the boundary conditions \eqref{e6.8}.

	Multiplying \eqref{e6.9} by $w$ and integrating over $(0,L)$, we obtain
	\begin{equation}\label{e6.10}
		a\|w\|^2+\frac{1}{2}(Dw(0))^2+\underbrace{(u_1^kDw,w)}_{I_1}+\underbrace{((u_1^k-u_2^k)Du_2,w)}_{I_2}=0.
	\end{equation}
	Integrating by parts and using \eqref{e3.1},\eqref{e6.8}, we get
	\begin{eqnarray*}
		I_1 & = & -\frac{1}{2}\int_{0}^{L}Du_1^k(x)w^2(x)dx\leq \frac{k}{2}\int_{0}^{L}|u_1(x)|^{k-1}|Du_1(x)||w(x)|^2dx\\
		& \leq & \frac{k}{2}\sup_{x\in(0,L)}|u_1(x)|^{k-1}\sup_{x\in(0,L)}|Du_1(x)|\|w\|^2\\ & \leq &  k2^{\frac{k-3}{2}}\|u_1\|_{H_0^1}^{k-1}\sup_{x\in(0,L)}|Du_1(x)|\|w\|^2.
	\end{eqnarray*}
	By \eqref{e3.1},\eqref{e5.6}, it follows that
	\begin{eqnarray*}
		|I_2| & \leq & 	\int_{0}^{L}|u_1^k(x)-u_2^k(x)||Du_2(x)||w(x)|dx \\
		& \leq & k2^{k-1}\sup_{x\in(0,L)}\{|u_1(x)|^{k-1}+|u_2(x)|^{k-1}\}\sup_{x\in(0,L)}|Du_2(x)|\|w\|^2\\
		& \leq & k2^{\frac{3(k-1)}{2}}(\|u_1\|_{H_0^1}^{k-1}+\|u_2\|_{H_0^1}^{k-1})\sup_{x\in(0,L)}|Du_2(x)|\|w\|^2.
	\end{eqnarray*}
	Substituting $I_1,I_2$ into \eqref{e6.10}, we get	
	\begin{eqnarray}\label{e6.11}
		a\|w\|^2-k2^{\frac{k-3}{2}}\|u_1\|_{H_0^1}^{k-1}\sup_{x\in(0,L)}|Du_1(x)|\|w\|^2 \notag\\
		-k2^{\frac{3(k-1)}{2}}(\|u_1\|_{H_0^1}^{k-1}+\|u_2\|_{H_0^1}^{k-1})\sup_{x\in(0,L)}|Du_2(x)|\|w\|^2\leq 0.
	\end{eqnarray}
	\begin{center}
		{\bf Regular case $\mathbf{1\leq k<4}$.}
	\end{center}
	By \eqref{e5.9},\eqref{e5.17}, 
	\begin{equation}\label{e6.12}
		\|D^3u_i\|\leq 2\|f\|+C_*^kC_2^{k+1}((1+x),f^2)^{\frac{k+1}{2}}, \,\, i=1,2.
	\end{equation}
	Making use of  \eqref{e3.3},\eqref{e5.9} and \eqref{e6.12}, we estimate 
	\begin{eqnarray*}
		\sup_{x\in (0,L)}|Du_i(x)| & \leq & K_1\|D^3u_i\|^{\frac{1}{2}}\|u_i\|^{\frac{1}{2}}+K_2\|u_i\|\\
		& \leq & \frac{K_1}{2}\|D^3u_i\|+\left(\frac{K_1}{2}+K_2\right)\|u_i\|\\
		& \leq &  \frac{K_1}{2}C_*^kC_2^{k+1}((1+x),f^2)^{\frac{k+1}{2}}+ K_3\|f\|\\
		& \leq & \frac{K_1}{2}C_*^kC_2^{k+1}((1+x),f^2)^{\frac{k+1}{2}}+ K_3((1+x),f^2)^{\frac{1}{2}},		
	\end{eqnarray*}
	where $K_3=\left(K_1+\frac{K_1}{2a}+\frac{K_2}{a}\right)$.
	Returning to \eqref{e6.11} and using \eqref{e5.13}, we find
	\begin{eqnarray*}
		a\|w\|^2-k(2^{\frac{k-3}{2}}+2^{\frac{3(k-1)}{2}})\frac{K_1}{2}C_*^kC_2^{2k}((1+x),f^2)^k\|w\|^2\notag\\
		-k(2^{\frac{k-3}{2}}+2^{\frac{3(k-1)}{2}})C_2^{k-1}K_3((1+x),f^2)^{\frac{k}{2}}\|w\|^2\leq 0.
	\end{eqnarray*}
	Assuming  $((1+x),f^2)^{\frac{1}{2}}\leq1,$ then $((1+x),f^2)^k\leq((1+x),f^2)^{\frac{k}{2}}$. Therefore
	$$
	\left(a-k(2^{\frac{k-3}{2}}+2^{\frac{3(k-1)}{2}})\left(\frac{K_1}{2}C_*^kC_2^{2k}+K_3C_2^{k-1}\right)((1+x),f^2)^{\frac{k}{2}}\right)\|w\|^2\leq 0.
	$$
	For fixed $k$ and $a$ assume that 
	\begin{equation}\label{e6.13}
		((1+x),f^2)^{\frac{1}{2}}<\min\left\{\left(\frac{1}{2aC_3}\right)^{\frac{4-k}{4k}},\left(\frac{a}{K_4}\right)^{\frac{1}{k}}\right\},
	\end{equation}
	where $K_4=k(2^{\frac{k-3}{2}}+2^{\frac{3(k-1)}{2}})(\frac{K_1}{2}C_*^k(a\beta)^{-k}+K_3(a\beta)^{-\frac{k-1}{2}})$. Then 
	$$C_2^{2k}<\left(\frac{1}{a\beta}\right)^k, \,\,\, C_2^{k-1}<\left(\frac{1}{a\beta}\right)^{\frac{k-1}{2}}$$ and
	$$
	\left(a-k(2^{\frac{k-3}{2}}+2^{\frac{3(k-1)}{2}})\left(\frac{K_1}{2}C_*^kC_2^{2k}+K_3C_2^{k-1}\right)((1+x),f^2)^{\frac{k}{2}}\right)> 0.
	$$
	
	This implies $\|w\|=0$ and uniqueness is proved for $l=1$ and $k<4$.
	\begin{center}
		{\bf Critical case $\mathbf{k=4}$.}
	\end{center}
	In this case, \eqref{e6.11} becomes
	\begin{eqnarray}\label{e6.14}
		a\|w\|^2-2^{\frac{5}{2}}\|u_1\|_{H_0^1}^3\sup_{x\in(0,L)}|Du_1(x)|\|w\|^2 \notag\\
		-2^{\frac{13}{2}}(\|u_1\|_{H_0^1}^3+\|u_2\|_{H_0^1}^3)\sup_{x\in(0,L)}|Du_2(x)|\|w\|^2\leq 0.
	\end{eqnarray}
	By \eqref{e5.9},\eqref{e5.20}, 
	\begin{equation}\label{e6.15}
		\|D^3u_i\|\leq 2\|f\|+C_*^4\left(\frac{1}{2a\gamma_1}\right)^{\frac{5}{2}}((1+x),f^2)^{\frac{5}{2}}, \,\, i=1,2,
	\end{equation}
	where $\gamma_1=\min\{\frac{a}{2},\frac{3}{2}-\frac{C_*^4}{6a^4}\|f\|^4\}$.
	Then \eqref{e3.3},\eqref{e5.9},\eqref{e6.15} implies 
	\begin{equation*}
		\sup_{x\in (0,L)}|Du_i(x)|\leq \frac{K_1}{2}C_*^4\left(\frac{1}{2a\gamma_1}\right)^{\frac{5}{2}}((1+x),f^2)^{\frac{5}{2}}+ K_3((1+x),f^2)^{\frac{1}{2}}.
	\end{equation*}
	Making use of \eqref{e5.20}, we rewrite \eqref{e6.14} as 
	\begin{eqnarray*}
		a\|w\|^2-(2^{\frac{5}{2}}+2^{\frac{15}{2}})\frac{K_1}{2}C_*^4\left(\frac{1}{2a\gamma_l}\right)^4((1+x),f^2)^4\|w\|^2\\
		-(2^{\frac{5}{2}}+2^{\frac{15}{2}})K_3\left(\frac{1}{2a\gamma_l}\right)^{\frac{3}{2}}((1+x),f^2)^2\|w\|^2\leq 0.
	\end{eqnarray*}
	Assuming $((1+x),f^2)^{\frac{1}{2}}\leq 1,$ then $((1+x),f^2)^4\leq ((1+x),f^2)^2$. This implies
	\begin{eqnarray*}
		a\|w\|^2-(2^{\frac{5}{2}}+2^{\frac{15}{2}})\frac{K_1}{2}C_*^4\left(\frac{1}{2a\gamma_l}\right)^4((1+x),f^2)^2\|w\|^2\\
		-(2^{\frac{5}{2}}+2^{\frac{15}{2}})K_3\left(\frac{1}{2a\gamma_l}\right)^{\frac{3}{2}}((1+x),f^2)^2\|w\|^2\leq 0.
	\end{eqnarray*}
	For a fixed $a,$ suppose that
	\begin{equation}\label{e6.16}
		((1+x),f^2)^{\frac{1}{2}}<\min\left\{\frac{\sqrt{3}a}{C_*},\left(\frac{a}{K_5}\right)^{\frac{1}{4}}\right\},
	\end{equation}
	where $K_5=(2^{\frac{5}{2}}+2^{\frac{15}{2}})(\frac{K_1}{2}C_*^4(2a\gamma_l)^{-4}+K_3(2a\gamma_l)^{-\frac{3}{2}})$. Then \eqref{e2.4} holds and 
	$$\left(a-(2^{\frac{5}{2}}+2^{\frac{15}{2}})\left(\frac{K_1}{2}C_*^4\left(\frac{1}{2a\gamma_l}\right)^4+K_3\left(\frac{1}{2a\gamma_l}\right)^{\frac{3}{2}}\right)((1+x),f^2)^2\right)>0.$$
	It follows that $\|w\|=0$ and uniqueness is proved for $l=1$ and $k=4$.
	
	This completes the proof of the uniqueness part of Theorem \ref{thm2}.
	
	To show continuous dependence of solutions, consider the case when $l\geq 2$ and $1\leq k<4l$. Let $f_1,f_2\in L^2(0,L)$ satisfy \eqref{e6.5} and $u_1$, $u_2$ be solutions of \eqref{e2.1}-\eqref{e2.2} with the right-hand  sides $f_1$ and $f_2$ respectively. Then, similarly to \eqref{e6.4}, $u_1-u_2$ satisfies the following inequality:
	\begin{equation*}
		\left(a-(2{^\frac{k-2}{2}}+2^{\frac{3k}{2}})k\tilde{C_2}^kM\right)\|u_1-u_2\|\leq \|f_1-f_2\|,
	\end{equation*}
	where
	$$M=\max\{((1+x), f_1^2)^{\frac{1}{2}}, ((1+x), f_2^2)^{\frac{1}{2}}\}$$
	and
	$$\tilde{C_2}=\frac{1}{\sqrt{\beta}}\left[C_3M^{\frac{4lk}{4l-k}}+\frac{1}{2a}\right]^{\frac{1}{2}}.$$
	Making use of \eqref{e6.5}, we obtain
	$$\|u_1-u_2\|\leq C_6 \|f_1-f_2\|$$
	with $C_6=\left(a-(2{^\frac{k-2}{2}}+2^{\frac{3k}{2}})k\tilde{C_2}^kM\right)^{-1}>0$. This proves the continuous dependence for $l\geq 2$ and $1\leq k<4l$. The other cases can be proved in a similar way taking $((1+x),{f_i}^2)^{\frac{1}{2}},\,\, i=1,2$ satisfying \eqref{e6.6}, \eqref{e6.13} and \eqref{e6.16}. Therefore the proof of the Theorem \ref{thm2} is complete.
\end{proof}

\end{document}